\documentclass[11pt,a4paper]{article}
\usepackage[pagebackref]{hyperref}
\usepackage[a4paper,margin=2.5cm]{geometry}
\usepackage{amssymb,amsmath,amsthm}
\usepackage{framed}
\DeclareMathOperator{\supp}{{\rm supp}}
\DeclareMathOperator{\argmax}{{\rm argmax}}

\newcommand{\argmin}{{\rm argmin\;}}

\newcommand{\E}{\mathbb{E}}
\newcommand{\p}{\mathbb{P}}

\newcommand{\Z}{\mathbb{Z}}

\newcommand*{\ind}[1]{\mathbf{1}_{\{#1\}}}
\newcommand*{\Ind}[1]{\mathbf{1}_{#1}}

\newcommand{\R}{\mathbb{R}}

\newtheorem{lemma}{Lemma}[section]
\newtheorem{theorem}[lemma]{Theorem}
\newtheorem{prop}[lemma]{Proposition}

\title{A note on the sample complexity of the Er-SpUD algorithm by Spielman, Wang and Wright for exact recovery of sparsely used dictionaries}

\author{Rados{\l}aw Adamczak\thanks{Institute of Mathematics, University of Warsaw, Banacha 2, 02-097 Warszawa, Poland. \mbox{E-mail: R.Adamczak@mimuw.edu.pl}}}

\begin{document}

\maketitle
\begin{abstract}
We consider the problem of recovering an invertible $n \times n$ matrix $A$ and a sparse $n \times p$ random matrix $X$ based on the observation of $Y = AX$ (up to a scaling and permutation of columns of $A$ and rows of $X$). Using only elementary tools from the theory of empirical processes we show that a version of the Er-SpUD algorithm by Spielman, Wang and Wright with high probability recovers $A$ and $X$ exactly, provided that $p \ge Cn\log n$, which is optimal up to the constant $C$.

\thanks{}
\end{abstract}
\section{Introduction}
Learning sparsely-used dictionaries has recently attracted considerable attention in connection to applications in machine learning, signal processing or computational neuroscience. In particular, two important fields of applications are \emph{dictionary learning} \cite{OlshausenField1996,Kreutz-DelgadoEtAl2003,BrucksteinDonohoElad2009,RubinsteinBrucksteinElad2010,YangEtAl2010} and \emph{blind source separation} \cite{ZibulevskyPearlmutter2001,GeorgievTheisCichocki2005}. We do not discuss these applications and refer the Reader to the aforesaid articles for details.

 Among many approaches to this problem a particularly successful one has been presented by Spielman, Wang and Wright \cite{SpielmanWangWrightCOLT,SpielmanWangWright13}, who considered the noiseless-invertible case:

\begin{framed} \noindent \textbf{The main problem:}

\medskip

\noindent  Consider an invertible $n \times n$ matrix $A$ and a random $n \times p$ sparse matrix $X$. Denote $Y = AX$. The objective is to reconstruct $A$ and $X$ (up to scaling and permutation of columns of $A$ and rows of $X$) based on the observable data $Y$.
\end{framed}

The Authors of \cite{SpielmanWangWright13} provide an algorithm which with high probability successfully recovers the matrices $A$ and $X$ up to rescaling and permutation of the columns of $A$ and rows of $X$, provided that $X$ is a sparse random matrix satisfying the following probabilistic assumptions.

\begin{framed}\noindent \textbf{Probabilistic model specification}

\begin{displaymath}
X_{ij} = \chi_{ij}R_{ij},
\end{displaymath}
where
\begin{itemize}
\item $\chi_{ij},R_{ij}$ are independent random variables,

\item $\chi_{ij}$ are Bernoulli distributed: $\p(\chi_{ij} = 1) = 1 - \p(\chi_{ij} = 0) = \theta$,

\item $R_{ij}$ are i.i.d., with mean zero and satisfy
\begin{align*}
\mu := \E|R_{ij}| & \ge 1/10,\\
\forall_{t> 0}\; \p(|R_{ij}| \ge t) &\le 2e^{-t^2/2}.
\end{align*}
\end{itemize}
\end{framed}

Following \cite{SpielmanWangWright13} we will say that matrices satisfying the above assumptions follow the Bernoulli-Subgaussian model with parameter $\theta$.

We remark that the constant 1/10 above is of no importance and has been chosen following \cite{SpielmanWangWright13} and \cite{LuhVu15}.

The approach of Spielman, Wang and Wright consists of two steps. At the first step (given by the Er-SpUD algorithm we describe below) one gathers $p/2$ candidates for the rows of $X$. The second, greedy step (Greedy algorithm, also described below) selects from the candidates the set of $n$ sparsest vectors, which form a matrix of rank $n$.

The algorithms work as follows:

\begin{framed}

\noindent \textbf{ER-SpUD(DC):  Exact  Recovery  of  Sparsely-Used  Dictionaries  using  the  sum  of
two  columns  of  Y   as  constraint  vectors.}

\begin{enumerate}
\item  \texttt{Randomly pair the columns of $Y$  into $p/2$ groups $g_j = \{Ye_{j_1}, Ye_{j_2}\}$.}
\item \texttt{For $j = 1,\ldots,p/2$\\
Let $r_j = Ye_{j_1} + Ye_{j_2}$, where $g_j = \{Ye_{j_1}, Ye_{j_2}\}$.\\
Solve $\min_w \|w^T Y\|_1$ subject to $r_j^T w = 1$, and set $s_j = w^T Y$.}
\end{enumerate}
\end{framed}

Above we use the convention that if $r_j=0$ (which happens with nonzero probability), and as a consequence the minimization problem has no solution, then we skip the corresponding step of the algorithm.

The second stage, described below, is run on the set $S$ of vectors $s_i$ returned at the first stage (for notational simplicity we relabel them if $r_j=0$ for some $j$). We use the standard notation that $\|x\|_0$ denotes the number of  nonzero coordinates of a vector $x$.
\begin{framed}
\noindent \textbf{Greedy: A  Greedy  Algorithm  to  Reconstruct  X  and  A.}
\begin{enumerate}
\item \texttt{  REQUIRE: $S = \{s_1,\ldots,s_T \} \subseteq \R^p$.}
\item\texttt{ For $i=1,\ldots,n$\\
REPEAT\\
$l \leftarrow \argmin_{s_l\in S} \| s_l\|_0$, breaking ties arbitrarily\\
$x_i = s_l$,
$S = S\setminus\{s_l\}$\\
UNTIL $rank([x_1,\ldots, x_i])= i$
\item Set $X = [x_1, \ldots, x_n]^T$  and $A = Y Y^T (XY^T)^{-1}$.}
\end{enumerate}
\end{framed}

In \cite{SpielmanWangWright13} it was proved that there exist positive constants $C,\alpha$, such that if
\begin{displaymath}
\frac{2}{n} \le \theta \le \frac{\alpha}{\sqrt{n}}
\end{displaymath}
and $p \ge Cn^2 \log^2 n$, then the ER-SpUD algorithm successfully recovers the matrices $A,X$ with probability at least $1 - \frac{1}{Cp^{10}}$. Note that the equation $Y = A'X'$ still holds if we set $A' = A\Pi \Lambda$ and $X' = \Lambda^{-1}\Pi^TX$ for some permutation matrix $\Pi$ and a nonsingular diagonal matrix $\Lambda$. Therefore, by recovery we mean that nonzero multiples of all the rows of $X$ are among the set $\{s_1,\ldots,s_{p/2}\}$ produced by the ER-SpUD(DC) algorithm. In \cite{SpielmanWangWright13} it is also proved that if $\p(R_{ij} = 0) = 0$, then for $p > Cn\log n$, with probability $1 - C'n\exp(-c\theta p)$ for any matrices $A',X'$ such that $Y = A' X'$ and $\max_{i} \|e_i^T X'\|_0 \le \max_{i} \|e_i^TX\|_0$ there exists a permutation matrix $\Pi$ and a nonsingular diagonal matrix $\Lambda$ such that $A' = A\Pi\Lambda$, $X' = \Lambda^{-1}\Pi^T X$. In fact, the Authors of \cite{SpielmanWangWright13} prove that with the above probability any row of $X$ is nonzero and has at most $(10/9)\theta p$ nonzero entries, whereas any linear combination of two or more rows of $X$ has at least $(11/9)\theta p$ entries.

In particular it follows that the Greedy algorithm will extract from the set $\{s_1,\ldots,s_{T}\}$ multiples of all $n$ rows of $X$ (note that all $s_j$'s are in the row space of $Y$ and thus also in the row space of $X$). Since, as one can prove, $X$ is with high probability of rank $n$, one easily proves that one can recover $A$ by the formula used in the 3rd step of the algorithm. We remark that in \cite{LuhVu15} Luh and Vu obtained the same results concerning sparsity of linear combinations of rows of $X$ without the assumptions about the symmetry of the variables $R_{ij}$.

Note also that for $\theta$ of the order $n^{-1}$, $p = Cn\log n$ is necessary for uniqueness of the solution in the sense described above, otherwise with significant probability some of the rows of $X$ may be zero, which means that some columns of $A$ do not influence the matrix $Y$.

In \cite{SpielmanWangWright13} it was also proved that if $p > Cn\log n$, $\theta > C'\sqrt{\frac{\log n}{n}}$, then with high probability the ER-SpUD algorithm does not recover any of the rows of $X$.

Spielman, Wang and Wright have conjectured that their algorithm works with high probability provided that $p > Cn\log n$ (which, as mentioned above is required for well-posedness of the problem).

Recently, Luh and Vu \cite{LuhVu15} have proved that the algorithm works for $p > Cn\log^4 n$, which differs from the conjectured number of samples just by a polylogarithmic factor.

In this note we will consider a modified version of the algorithm with a slightly different first stage. Namely, instead of using only $p/2$ pairs of columns of $Y$, we will use all $\binom{p}{2}$ pairs. For fixed $p$ it clearly increases the time complexity of the algorithm (which however remains polynomial), but the advantage of this modification is the possibility of proving that it requires only $p = Cn\log n$ to recover $X$ and $A$ with high probability, which as explained above is optimal. More specifically, we will consider the following algorithm.

\begin{framed}
\noindent \textbf{Modified ER-SpUD(DC):  Exact  Recovery  of  Sparsely-Used  Dictionaries  using  the  sum  of
two  columns  of  Y   as  constraint  vectors.}

\texttt{For $i = 1,\ldots,p-1$\\
\phantom{aa} For $j= i+1,\ldots,p$\\
\phantom{aaaa}Let $r_{ij} = Ye_{i} + Ye_{j}$\\
\phantom{aaaa}Solve $\min_w \|w^T Y\|_1$ subject to $r_{ij}^T w = 1$, and set $s_{ij} = w^T Y$.}
\end{framed}

The final step of the recovery algorithm is again a greedy selection of the sparsest vectors among the candidates collected at the first step. As before, under the assumption $\p(R_{ij} = 0) = 0$, the greedy procedure successfully recovers $X$ and $A$, provided that multiples of all the rows of $X$ are present among the input set $S$.

The main result of this note is

\begin{theorem}\label{thm:main}
There exist absolute constants $C,\alpha \in (0,\infty)$ such that if
\begin{displaymath}
\frac{2}{n} \le \theta \le \frac{\alpha}{\sqrt{n}}
\end{displaymath}
and $X$ follows the Bernoulli-Subgaussian model with parameter $\theta$, then for $p \ge Cn\log n$, with probability at least $1 - 1/p$ the modified ER-SpUD algorithm successfully recovers all the rows of $X$, i.e. multiples of all the rows of $X$ are present among the vectors $s_{ij}$ returned by the algorithm.
\end{theorem}

\paragraph{Remark} Very recently in \cite{SunQingWright15}, Sun, Qing and Wright proposed an algorithm with polynomial sample complexity, which recovers well conditioned dictionaries under the assumption that the variables $R_{ij}$ are i.i.d. standard Gaussian and $\theta \le 1/2$, thus allowing for the first time for a linear number of nonzero entries per column of the matrix $X$. Their novel approach is based on non-convex optimization. The sample complexity of the algorithms in \cite{SunQingWright15} is however higher then for the Er-SpUD algorithm; as mentioned by the Authors, numerical simulations suggest that it is at least $p = \Omega(n^2\log n)$ even in the case of orthogonal matrix $A$. The Authors of \cite{SunQingWright15} conjecture that algorithms with sample complexity $p = O(n\log n)$ should be possible also for large $\theta$.

\section{Proof of Theorem \ref{thm:main}\label{sec:proof}}

We will follow the general approach presented in \cite{SpielmanWangWright13} and \cite{LuhVu15}. The main new part of the argument is an improved bound on the sample complexity for empirical approximation of first moments of arbitrary marginals of the columns of the matrix $X$, given in Proposition \ref{prop:empirical-process} below. So as not to reproduce technical and lengthy parts of the original proof, we organize this section as follows. First, we present the crucial Proposition \ref{prop:empirical-process} and provide a brief discussion of its mathematical content. Next, we present an overview of the main steps in the proof scheme of \cite{SpielmanWangWright13}. For parts of the proof not related to Proposition \ref{prop:empirical-process} or to the modification of the algorithm considered here, we only indicate the relevant statements from \cite{SpielmanWangWright13}, while for the part involving the use of Proposition \ref{prop:empirical-process} and for the conclusion of the proof we provide the full argument. Proposition \ref{prop:empirical-process} is proved in Section \ref{sec:proof-of-proposition}.

Below by $e_1,\ldots,e_N$ we will denote the standard basis in $\R^N$ for various choices of $N$ (in particular for $N = n$ and $N=p$). The value of $N$ will be clear from the context and so this should not lead to ambiguity.

By $B_1^n$ we will denote the unit ball in the space $\ell_1^n$, i.e. $B_1^n = \{x\in \R^n \colon \|x\|_1\le 1\}$, where for $x = (x(1),\ldots, x(n))$, $\|x\|_1 = \sum_{i=1}^n |x(i)|$. The coordinates of a vector $x$ will be denoted by $x(i)$ or if it does not interfere with other notation (e.g. for indexed families of vectors) simply by $x_i$. Again, the meaning of the notation will be clear from the context.

\begin{prop}\label{prop:empirical-process}
Let $U_1,U_2,\ldots, U_p,\chi_1,\ldots,\chi_p$ be independent random vectors in $\R^n$. Assume that for some constant $M$ and all $1\le i \le p$, $1\le j\le n$,
\begin{align}\label{eq:exp_integr}
\E e^{|U_{i}(j)|/M} \le 2
\end{align}
and
\begin{displaymath}
\p( \chi_i(j) = 1) = 1 - \p(\chi_i(j) = 0) = \theta.
\end{displaymath}
Define the random vectors $Z_1,\ldots,Z_p$ with the equality $Z_i(j) = U_i(j)\chi_i(j)$ for $1\le i \le p$, $1\le j\le n$ and consider the random variable
\begin{align}
W := \sup_{x \in B_1^n} \Big|\frac{1}{p}\sum_{i=1}^p (|x^T Z_i| - \E |x^T Z_i|) \Big|.
\end{align}
Then, for some universal constant $C$ and every $q \ge \max(2,\log n)$,
\begin{align}
\|W\|_{q} \le \frac{C}{p}(\sqrt{p\theta q} + q  )M
\end{align}
and as a consequence
\begin{align}
\p\Big(W \ge \frac{Ce}{p}(\sqrt{p\theta q} + q )M\Big) \le e^{-q}.
\end{align}
\end{prop}

The above proposition can be considered a quantitative version of the uniform law of large numbers for linear functionals $x^TZ$ indexed by the unit sphere in the space $\ell_1^n$. As such it is a classical object of study in the theory of empirical processes. The proof we give uses only Bernstein's inequality (see e.g. \cite{vanderVaartWellner1996}) and Talagrand's contraction principle \cite{LedouxTalagrand1991}, which in a somewhat similar context was applied e.g. in \cite{Mendelson2008,ALPT2010}.

Let us also remark that in the above proposition we do not require independence between components of the random vectors $U_i$ or $\chi_i$ for fixed $i$, but just independence between the random vectors $U_i,\chi_i, i=1,\ldots,p$.

\subsection{Main steps of the proof of Theorem \ref{thm:main}} As announced, we will now present an outline of the proof of Theorem \ref{thm:main}, indicating which steps differ from the original argument in \cite{SpielmanWangWright13}.

\medskip
\noindent{\textbf{Step 1.} A change of variables.}
\medskip

Recall that $r_{ij}$ are sums of two columns of the matrix $Y$. At the first step of the proof, instead of looking at the original optimization problem
\begin{align}\label{eq:optimization-problem}
\textrm{minimize } \|w^T Y\|_1 \textrm{ subject to } r_{ij}^T w = 1
\end{align}
one performs a change of variables $z = A^Tw$, $b_{ij} = A^{-1}r_{ij}$, arriving at the optimization problem
\begin{align}\label{eq:changed-problem}
\textrm{minimize } \|z^T X\|_1 \textrm{ subject to } b_{ij}^T z = 1.
\end{align}

Note that one cannot solve \eqref{eq:changed-problem} since it involves the unknown matrices $X$ and $A$. The goal of the subsequent steps is to prove that with probability separated from zero the solution $z_\ast$ of \eqref{eq:changed-problem} is a multiple of one of the basis vectors $e_1,\ldots,e_n$, say $z_\ast = \lambda e_k$. This means that $w_\ast^T Y =  z_\ast^T X = \lambda e_k^T X$, i.e. \eqref{eq:optimization-problem} recovers the $k$-th row of $X$ up to scaling.

\medskip
\noindent{\textbf{Step 2.} The solution $z_\ast$ satisfies $\supp(z_\ast) \subseteq \supp(b_{ij})$.}
\medskip

At this step we prove the following lemma, which is a counterpart of Lemma 11 in \cite{SpielmanWangWright13}. It is weaker in that we do not consider arbitrary vectors $b_{ij}$, but only sums of two distinct columns of $X$ (which is enough for the application in the proof of Theorem \ref{thm:main}). On the other hand it works already for $p > Cn\log n$ and not for $p > Cn^2\log n$ as the original lemma from \cite{SpielmanWangWright13}.

\begin{lemma}\label{le:4.2}

For $1\le i< j \le p$, define  $b_{ij} = Xe_i + Xe_j$, $I_{ij} = (\supp Xe_i) \cup (\supp Xe_j)$.
There exist numerical constants $C,\alpha> 0$ such that if $2/n \le \theta \le \alpha/\sqrt{n}$ and $p > Cn\log n$, then with probability at least $1-p^{-2}$ the random matrix $X$ has the following property:

(P1)  For every $1\le i < j \le p$ either $|I_{ij}| \in\{0\} \cup (1/(8\theta),n]$ or every solution $z_\ast$ to the optimization problem \eqref{eq:changed-problem} satisfies $\supp z_\ast \subseteq I_{ij}$.
\end{lemma}

To prove the above lemma, one first shows a counterpart of Lemma 16 in \cite{SpielmanWangWright13}.

\begin{lemma}\label{le:4.16} For any $1\le j \le p$, if $Z = (\chi_{1j}R_{1j},\ldots,\chi_{nj}R_{nj})$, then for all $v \in \R^n$,
\begin{displaymath}
\E |v^T Z| \ge \frac{\mu}{8}\sqrt{\frac{\theta}{n}}\|v\|_1.
\end{displaymath}
\end{lemma}
\begin{proof}
Let $\varepsilon_1,\ldots,\varepsilon_n$ be a sequence of i.i.d. Rademacher variables, independent of $Z$. By standard symmetrization inequalities (see e.g. Lemma 6.3. in \cite{LedouxTalagrand1991}),
\begin{displaymath}
\E|v^T R| = \E\Big|\sum_{i=1}^n v_i \chi_{ij}R_{ij}\Big| \ge \frac{1}{2}\E\Big|\sum_{i=1}^n v_i \varepsilon_i \chi_{ij}R_{ij}\Big|.
\end{displaymath}
The random variables $\varepsilon_i R_{ji}$ are symmetric and $\E|\varepsilon_i R_{ij}| = \mu$, so by Lemma 16 from \cite{SpielmanWangWright13}, the right-hand side above is bounded from below by $\frac{\mu}{8}\sqrt{\frac{\theta}{n}}\|v\|_1$.
\end{proof}

The next lemma is an improvement of Lemma 17 in \cite{SpielmanWangWright13}, which is crucial for obtaining Lemma \ref{le:4.2}.

\begin{lemma}\label{le:4.17} There exists an absolute constant $C$, such that the following holds for $p > Cn\log n$. Let $S \subseteq \{1,\ldots,p\}$ be a fixed subset of size $|S| < \frac{p}{4}$. Let $X_S$ be the submatrix of $X$, obtained by a restriction of $X$ to the columns indexed by $S$. With probability at least $1- p^{-8}$, for any $v \in \R^n$,
\begin{displaymath}
\|v^T X\|_1 - 2\|v^T X_S\|_1 > \frac{p\mu}{32}\sqrt{\frac{\theta}{n}} \|v\|_1.
\end{displaymath}
\end{lemma}

\begin{proof}\label{le:17} Note first that by increasing the set $S$, we increase $\|v^TX_S\|_1$, so without loss of generality we can assume that $|S| = \lfloor p/4\rfloor$.
Apply Proposition \ref{prop:empirical-process} with the vectors $U_j = (R_{1j},\ldots,R_{nj})$ and $\chi_j = (\chi_{1j},\ldots,\chi_{nj})$ and $q=8\log p$. Note that our integrability assumptions on $R_{ij}$ imply \eqref{eq:exp_integr} with $M$ being a universal constant. Therefore, for some absolute constant $C$ and $p \ge Cn\log n$, with probability at least $1 - p^{-8}$ we have
\begin{align*}
\sup_{v \in B_1^n} \Big|\|v^T X\|_1 - \E \|v^T X\|_1 \Big| &\le C(\sqrt{p\theta \log p} + \log p) \le 2C\sqrt{p\theta \log p},\\
\sup_{v \in B_1^n} \Big|\|v^T X_S\|_1 - \E \|v^T X_S\|_1 \Big| &\le 2C\sqrt{p\theta \log p},
\end{align*}
where we used that for $C$ sufficiently large, $p/\log p \ge n \ge 1/\theta$.

Thus, by homogeneity, for all $v \in \R^n$,
\begin{align*}
\Big|\|v^T X\|_1 - \E \|v^T X\|_1 \Big| & \le 2C\sqrt{\theta p\log p}\|v\|_1,\\
\Big|\|v^T X_S\|_1 - \E \|v^T X_S\|_1 \Big|&\le 2C\sqrt{\theta p\log p}\|v\|_1.
\end{align*}
In particular this means that (using the notation of Proposition \ref{prop:empirical-process})
\begin{align*}
\|v^T X\|_1 &\ge  \E \|v^T X\|_1 - 2C\sqrt{\theta p \log p}\|v\|_1 = p \E |v^T Z_1| - 2C\sqrt{\theta p\log p}\|v\|_1,\\
2\|v^T X_S\|_1 &\le 2\E \|v^T X_S\|_1 +  4C\sqrt{\theta p\log p}\|v\|_1 = 2|S| \E |v^T Z_1| + 4C\sqrt{\theta p \log p}\|v\|_1,
\end{align*}
and so
\begin{displaymath}
\|v^T X\|_1 - 2\|v^T X_S\|_1 \ge (p - 2|S|) \E|v^T Z_1| - 6C\sqrt{\theta p \log p}\|v\|_1.
\end{displaymath}

Now, by Lemma \ref{le:4.16} and the assumed bound on the cardinality of $S$, we get
\begin{displaymath}
\|v^T X\|_1 - 2\|v^T X_S\|_1 \ge \Big(\frac{p\mu}{16}\sqrt{\frac{\theta}{n}}  - 6C\sqrt{\theta p \log p}\Big)\|v\|_1 >  \frac{p\mu}{32}\sqrt{\frac{\theta}{n}} \|v\|_1
\end{displaymath}
for $p > C'n\log n$, where $C'$ is another absolute constant.
\end{proof}

We are now in position to prove Lemma \ref{le:4.2}.

\begin{proof}[Proof of Lemma \ref{le:4.2}] We will show that for each $1\le i< j \le p$ the probability that $0< |I_{ij}| \le 1/(8\theta)$ and there exists a solution to \eqref{eq:changed-problem} not supported on $I_{ij}$ is bounded from above by $1/p^4$. By the union bound over all $i< j$, this implies the lemma.

Fix $i,j$ and let $S = \{l \in [p]\colon \exists_{k \in I_{ij}} X_{kl} \neq 0\}$. Denote by $\mathcal{F}_1$ the $\sigma$-field generated by $Xe_i$ and $Xe_j$.
Then $\mathcal{A} = \{0< |I_{ij}| \le 1/(8\theta)\} \in \mathcal{F}_1$. By independence, for each $k \notin \{i,j\}$, on the event $\mathcal{A}$,
\begin{displaymath}
\p(k \in S|\mathcal{F}_1) \le 1 - (1-\theta)^{|I_{ij}|} \le 1 - e^{-2\theta|I_{ij}|} \le 1 - e^{-\frac{1}{4}} < \frac{1}{4},
\end{displaymath}
where the second inequality holds if $\alpha$ is sufficiently small.

Thus, by independence of columns of $X$ and Hoeffding's inequality,
\begin{align}\label{eq:first}
\p\Big( |S\setminus \{i,j\}| \le \frac{p}{4}\Big|\mathcal{F}_1\Big) \ge 1  - e^{-cp}
\end{align}
for some universal constant $c > 0$.
Let $z_\ast$ be any solution of \eqref{eq:changed-problem} and denote by $z_0$ its orthogonal projection on $\R^{I_{ij}} = \{x\in \R^n\colon x_k = 0 \textrm{ for } k \notin I_{ij}\}$. Set also $z_1 = z_\ast - z_0$ and let $X_S, X_{S^c}$ be the matrices obtained from $X$ by selecting the columns labeled by $S$ and $S^c=[p]\setminus S$ respectively. By the triangle inequality, and the fact that $z_0^TX_{S^c} = 0$ , we get
\begin{align*}
\|z_\ast^T X\|_1 &= \|(z_0^T+z_1^T)X_S\|_1 +\|(z_0^T+z_1^T)X_{S^c}\|_1 \\
&\ge \|z_0^T X_S\|_1 - \|z_1^TX_S\|_1 +\|z_1^T X\|_1 - \|z_1^T X_{S}\|_1\\
&= \|z_0^T X\|_1 + (\|z_1^T X\|_1 - 2\|z_1^T X_{S}\|_1).
\end{align*}

Denote now by $X'$ the $|I_{ij}^c|\times (p-2)$ matrix obtained by restricting $X$ to the rows from $I_{ij}^c$ and columns from $[p]\setminus\{i,j\}$. Set also $S' = S\setminus\{i,j\}$. If, slightly abusing the notation, we identify $z_1$ with a vector from $\R^{|I_{ij}^c|}$, we have
\begin{displaymath}
\|z_1^T X\|_1 - 2\|z_1^T X_{S}\|_1 = \|z_1^T X'\|_1 - 2\|z_1^T X'_{S'}\|_1,
\end{displaymath}
where we used the fact that $z_1^T Xe_i = z_1^T X e_j = 0$.

Denote by $\mathcal{F}_2$ the $\sigma$-field generated by $Xe_i,Xe_j$ and the rows of $X$ labeled by $I_{ij}$ (note that $I_{ij}$ is itself random, but this will not be a problem in what follows). The random set $S$ is measurable with respect to $\mathcal{F}_2$. Moreover, due to independence and identical distribution of the entries of $X$, conditionally on $\mathcal{F}_2$ the matrix $X'$ still follows the Bernoulli-Subgaussian model with parameter $\theta$. Therefore, by Lemma \ref{le:4.17}, if $C$ is large enough, then on $\{|S'| \le p/4\}$ we have
\begin{align*}
\p\Big(\textrm{for all} \; v\in \R^{|I_{ij}^c|}:\quad\|v^T X'\|_1 - 2\|v^T X'_{S'}\|_1 \ge \frac{p\mu}{32}\sqrt{\frac{\theta}{|I_{ij}^c|}}\|v\|_1\Big|\mathcal{F}_2\Big) \ge 1 - p^{-8}.
\end{align*}

Note that by the definition of $z_0$, we have $b_{ij}^T z_0 = b_{ij}^T z = 1$, therefore $z_0$ is a feasible candidate for the solution of the optimization problem \eqref{eq:changed-problem}. Thus, we have $\|z_1^T X'\|_1 - 2\|z_1^T X'_{S'}\|_1 \le 0$ and as a consequence, on the event $\{|S'|\le p/4\}$,
\begin{align}\label{eq:second}
\p(\textrm{for some solution $z_\ast$ to \eqref{eq:changed-problem},}\; z_1 \neq 0|\mathcal{F}_2) \le p^{-8}.
\end{align}
Thus, denoting $\mathcal{B} = \{\textrm{for some solution $z_\ast$ to \eqref{eq:changed-problem},}\;z_1 \neq 0 \; \textrm{and}\; 0< |I_{ij}^c| < 1/(8\theta)\}$, we get by \eqref{eq:first} and \eqref{eq:second},
\begin{align*}
\p(\mathcal{B}) &\le \p(\mathcal{B}\cap\{|S'|> p/4\}) + \E \p(\mathcal{B}|\mathcal{F}_2)\ind{|S'|\le p/4}\\
&\le \E \p(|S'|> p/4|\mathcal{F}_1)\Ind{\mathcal{A}} + p^{-8}\\
&\le e^{-cp} + p^{-8} \le p^{-4}
\end{align*}
for $p > Cn\log n$ with a sufficiently large absolute constant $C$.
\end{proof}

\medskip
\noindent{\textbf{Step 3.} With high probability $z_\ast = \lambda e_k$ for $k = \argmax_{1\le l\le n} |b_{ij}(l)|$.}
\medskip

At this step one proves the following lemma (Lemma 12 in \cite{SpielmanWangWright13}). Since no changes with respect to the original argument are required (we do not use Proposition \ref{prop:empirical-process} here), we do not reproduce the proof and refer the Reader to \cite{SpielmanWangWright13} for details. We remark that although the lemma is formulated in \cite{SpielmanWangWright13} for symmetric variables, the symmetry assumption is not used in its proof.

Below, by $|b|^\downarrow_1 \ge |b|^\downarrow_2 \ge \ldots \ge |b|^\downarrow_n$, we denote the nonincreasing rearrangement of the sequence $|b_1|,\ldots,|b_n|$, while for $J \subseteq [n]$, $X^J$ denotes the matrix obtained from $X$ by selecting the rows indexed by the set $J$.

\begin{lemma}\label{le:12} There exist two positive constants $c_1,c_2$ such that the following holds. For any $\gamma > 0$ and $s \in \Z_+$, such that $\theta s < \gamma/8$ and $p$ such that
\begin{displaymath}
p \ge \max\Big\{\frac{c_1s\log n}{\theta\gamma^2}, n\Big\},\quad \textrm{and} \quad \frac{p}{\log p} \ge \frac{c_2}{\theta\gamma^2},
\end{displaymath}
with probability at least $1 - 4p^{-10}$, the random matrix $X$ has the following property.

(P2) For every $J\subseteq [n]$ with $|J| = s$ and every $b \in \R^s$, satisfying $\frac{|b|^\downarrow_1}{|b|^\downarrow_2}\le 1-\gamma$, the solution to the restricted problem
\begin{align}\label{eq:restricted-optimization}
\textrm{minimize } \|z^T X^J\|_1 \textrm{ subject to } b^T z = 1,
\end{align}
is unique, 1-sparse, and is supported on the index of the largest entry of $b$.
\end{lemma}

\medskip
\noindent{\textbf{Step 4.} Conclusion of the proof.}

\medskip

Set $s = 12\theta n +1$. Our first goal is to prove that with probability at least $1-1/p^2$, for all $k \in [n]$, there exist $i,j\in [p]$, $i\neq j$ such that the vector $b = Xe_i + Xe_j$ satisfies the assumptions of Lemma \ref{le:12}, $|b|^\downarrow_1 = |b_k|$ and $I_{ij} := (\supp Xe_i)\cup(\supp Xe_j)$ satisfies $0<  |I_{ij}| \le 1/(8\theta)$, which will allow us to take advantage of Lemma \ref{le:4.2}.

Note that we have
\begin{displaymath}
\E R_{ij}^2 \le 4\int_0^\infty te^{-t^2/2}dt = 4.
\end{displaymath}
Since $\E|R_{ij}| = \mu \ge \frac{1}{10}$, by the Paley-Zygmund inequality (see e.g. Corollary 3.3.2. in \cite{delaPenaGine1999}), we have
\begin{displaymath}
\p(|R_{ij}| \ge \frac{1}{20}) \ge \frac{3}{4}\frac{(\E |R_{ij}|)^2}{\E R_{ij}^2} \ge c_0
\end{displaymath}
for some universal constant $c_0>0$. In particular $\p(|R_{ij}| = 0) < 1 - \frac{c_0}{2}$.
Let $q$ be any $(1-c_0/(2s))$-quantile of $|R_{ij}|$, i.e. $\p(|R_{ij}| \le q) \ge (1-c_0/(2s))$ and $\p(|R_{ij}| \ge q) \ge c_0/(2s)$. In particular, since $s \ge 1$, we get $q > 0$.
We have $\p(R_{ij} \ge q) \ge c_0/(4s)$ or $\p(R_{ij} \le -q) \ge c_0/(4s)$. Let us assume that $\p(R_{ij} \ge q) \ge c_0/(4s)$, the other case is analogous.

Define the event $\mathcal{E}_{ki}$ as
\begin{align*}
\mathcal{E}_{ki} = \Big\{\chi_{ki} = 1,\;
|\{r \in [n]\setminus\{k\}\colon \chi_{ri} = 1\}| \le (s-1)/2,\;
R_{ki} \ge q,\;
\forall_{r\neq k}\, \chi_{ri} = 1 \implies |R_{ri}| \le q\Big\}
\end{align*}

We will assume that $p \ge 2Cn\log n$ for some numerical constant $C$ to be fixed later on. For $k\in [n]$, consider the events
\begin{displaymath}
\mathcal{A}_k = \bigcup_{1\le i\le \lfloor p/2\rfloor} \mathcal{E}_{ki}
\end{displaymath}
and
\begin{displaymath}
\mathcal{B}_i = \bigcup_{1\le i\le \lfloor p/2\rfloor} \bigcup_{\lfloor p/2\rfloor < j \le p} \mathcal{E}_{ki}\cap\mathcal{E}_{kj}\cap\Big\{\{l\in [n]\colon \chi_{li} = \chi_{lj} = 1\} = \{k\}\Big\}.
\end{displaymath}
We will first show that for all $k\in [n]$,
\begin{align}\label{eq:first_property}
\p(\mathcal{A}_k) \ge 1 - \frac{1}{p^4},
\end{align}
which we will use to prove that
\begin{align}\label{eq:second_property}
\p(\mathcal{B}_k) \ge 1 -\frac{1}{p^3}.
\end{align}
Let us start with the proof of \eqref{eq:first_property}. Set $\mathcal{B}_{ki} =\{
|\{r \in [n]\setminus\{k\}\colon \chi_{rk} = 1\}| \le (s-1)/2\}$. By independence we have
\begin{align*}
\p(\mathcal{E}_{ki}) &= \p(\chi_{ki} = 1)\p(R_{ki} \ge q)\p(\mathcal{B}_{ki})\p(\forall_{r\neq k}\, \chi_{ri} = 1 \implies |R_{ri}| \le q|\mathcal{B}_{ki})\\
&\ge \theta \frac{c_0}{4s}\Big(1-\frac{2\theta (n-1)}{s-1}\Big)\Big(1-\frac{c_0}{2s}\Big)^{(s-1)/2},
\end{align*}
where to estimate $\p(B_{ki})$ we used Markov's inequality. The right hand side above is bounded from below by $c_1/n$ for some universal constant $c_1$. Therefore if the constant $C$ is large enough, we obtain
\begin{displaymath}
\p\Big(\bigcap_{1\le i\le \lfloor p/2\rfloor} \mathcal{E}_{ki}^c\Big) \le \Big(1 - \frac{c_1}{n}\Big)^{\lfloor p/2\rfloor} \le \exp(-c_1 p/(4n)) \le \exp(-4\log p) = \frac{1}{p^4},
\end{displaymath}
where we used the inequality $p/\log p \ge 16c_1^{-1}n$ for $p \ge Cn\log n$. We have thus established \eqref{eq:first_property}.

Let us now pass to \eqref{eq:second_property}. Denote by $\mathcal{F}_1$ the $\sigma$-field generated by $\chi_{ki},R_{ki}, k \in [n], 1\le i \le \lfloor p/2\rfloor$.

For  $\omega \in \mathcal{A}_k$ define $i_{\min}(\omega) = \min \{1\le i \le \lfloor p/2\rfloor \colon \omega \in \mathcal{E}_{ki}\}$. Note that on $\mathcal{A}_k$,
\begin{align*}
\p(\mathcal{B}_k|\mathcal{F}_1) &\ge \p \Big(\bigcup_{\lfloor p/2\rfloor < j \le p} (\mathcal{E}_{kj} \cap \Big\{\{l\in [n]\colon \chi_{li_{\min}} = \chi_{lj} = 1\}=\{k\}\Big\}\Big|\mathcal{F}_1\Big)
\end{align*}
Define
\begin{displaymath}
\mathcal{C}_{kj} = \{|\{r \in [n]\setminus\{k\}\colon \chi_{rj} = 1\}| \le (s-1)/2,\}\cap \Big\{\{l\in [n]\colon \chi_{li_{\min}} = \chi_{lj} = 1\}=\{k\}\Big\}.
\end{displaymath}
Similarly as in the argument leading to \eqref{eq:first_property}, for fixed $j$, using the independence of the variables $\chi_{lm},R_{lm}$ we obtain
\begin{align*}
&\p\Big(\mathcal{E}_{kj} \cap \Big\{\{l\in [n]\colon \chi_{li_{\min}} = \chi_{lj} = 1\}=\{k\}\Big\}\Big|\mathcal{F}_1\Big)\\
&= \p(R_{kj} \ge q)\E\Big(\Ind{\mathcal{C}_{kj}}\p(\forall_{r\neq k}\, \chi_{rj} = 1 \implies |R_{rj}| \le q|\mathcal{C}_{kj},\mathcal{F}_1)|\mathcal{F}_1\Big)\\
&\ge \p(R_{kj} \ge q) \E \Big(\Ind{\mathcal{C}_{kj}}\Big(1-\frac{c_0}{2s}\Big)^{\frac{s-1}{2}}\Big|\mathcal{F}_1\Big)\\
&= \p(R_{kj} \ge q) \Big(1-\frac{c_0}{2s}\Big)^{\frac{s-1}{2}} \p(\mathcal{C}_{kj}|\mathcal{F}_1) \\
&\ge \frac{c_0}{4s} \Big(1-\frac{c_0}{2s}\Big)^\frac{s-1}{2} \times \\
&\Big(\p\Big(\{l\in [n]\colon \chi_{li_{\min}} = \chi_{lj} = 1\}=\{k\}\Big|\mathcal{F}_1\Big) - \p\Big(\chi_{kj} = 1,\;|\{r \in [n]\setminus\{k\}\colon \chi_{rj} = 1\}| > \frac{s-1}{2}\Big|\mathcal{F}_1\Big)\Big)\\
&\ge \frac{c_0}{4s}\Big(1-\frac{c_0}{4s}\Big)^\frac{s-1}{2} \Big(\theta (1-\theta)^{(s-1)/2} - \theta \frac{2\theta(n-1)}{s-1}\Big).
\end{align*}
Now recall that $\theta \le \frac{\alpha}{n}$ for some universal constant $\alpha$. If $\alpha$ is small enough then $1 -\theta \ge e^{-2\theta}$ and
\begin{displaymath}
(1-\theta)^{(s-1)/2} \ge e^{-\theta (s-1)} = e^{-12\theta^2 n}\ge e^{-12\alpha^2} \ge \frac{1}{3}.
\end{displaymath}
Since $\frac{2\theta(n-1)}{s-1} \le \frac{1}{6}$, this implies that
\begin{displaymath}
\p\Big(\mathcal{E}_{kj} \cap \Big\{\{l\in [n]\colon \chi_{lj_{\min}} = \chi_{lj} = 1\}=\{k\}\Big\}\Big|\mathcal{F}_1\Big) \ge \frac{c_2}{n}
\end{displaymath}
for some positive universal constant $c_2$. Since the events $\mathcal{E}_{kj} \cap \Big\{\{l\in [n]\colon \chi_{lj_{\min}} = \chi_{lk} = 1\}=\{k\}\Big\}$, $\lfloor p/2\rfloor < k \le p$ are conditionally independent, given $\mathcal{F}_1$, we obtain that on $\mathcal{A}_k$,
\begin{displaymath}
\p(\mathcal{B}_k^c|\mathcal{F}_1) \le \Big(1- \frac{c_2}{n}\Big)^{\lfloor p/2\rfloor} \le \frac{1}{p^4},
\end{displaymath}
provided $C$ is a sufficiently large universal constant.
Now, using \eqref{eq:first_property}, we get
\begin{displaymath}
\p(\mathcal{B}_k) \ge \E \Ind{\mathcal{A}_k}\p(\mathcal{B}_k) \ge \p(A_k)\Big(1-\frac{1}{p^4}\Big) \ge \Big(1-\frac{1}{p^4}\Big)^2 \ge 1 - \frac{1}{p^3},
\end{displaymath}
proving \eqref{eq:second_property}.

Taking the union bound over $k \in [n]$, we get
\begin{displaymath}
\p(\bigcap_{1\le k\le n} \mathcal{B}_k) \ge 1 - \frac{1}{p^2}.
\end{displaymath}

Set $\gamma = 1/2$ and observe that if $C$ is large enough and $\alpha$ small enough, then the assumptions of Lemma \ref{le:4.2} and Lemma \ref{le:12} are satisfied. Moreover $s = 12\theta n + 1 \le \frac{1}{8\theta}$. Recall the properties P1 and P2 considered in the said lemmas. Consider the event $\mathcal{A} = \bigcap_{1\le k\le n} \mathcal{B}_k \cap \{\textrm{properties P1 and P2 hold}\}$ and note that $\p(\mathcal{A}) \ge 1 - \frac{1}{p}$. On the event $\mathcal{A}$, for every $k$, there exist $1\le i< j\le p$, such that
\begin{itemize}
\item $1\le |I_{ij}| \le s \le 1/(8\theta)$,

\item the largest entry of $b$ (in absolute value) equals $b_k\ge 2q > 0$ whereas the remaining entries do not exceed $q$,
\end{itemize}

In particular, by property P1 we obtain that any solution $z_\ast$ to the problem \eqref{eq:changed-problem} satisfies $\supp z_\ast \subseteq I_{ij}$. Therefore for some (any) $J \supseteq I_{ij}$ with $|J| = s$, we obtain (identifying vectors supported on $J$ with their restrictions to $J$), that $z_\ast$ is in fact a solution to the restricted problem \eqref{eq:restricted-optimization} with $b=b_{ij}$, which by property P2 implies that $z_\ast = \lambda e_k$ for some $\lambda \neq 0$.

According to the discussion at the beginning of Step 1, this means that the solution $w_\ast$ to \eqref{eq:optimization-problem} satisfies $w_\ast^TY= \lambda e_k^T X$, i.e. the algorithm, when analyzing the vector $b_{ij}$, will add a multiple of the $k$-th row of $X$ to the collection $S$.

This ends the proof of Theorem \ref{thm:main}.

\section{Proof of Proposition \ref{prop:empirical-process} \label{sec:proof-of-proposition}}
The first tool we will need is the classical Bernstein's inequality (see e.g. Lemma 2.2.11 in \cite{vanderVaartWellner1996}).
\begin{lemma}[Bernstein's inequality]
Let $Y_1,\ldots, Y_p$ be independent mean zero random variables such that for some constants $M,v$ and every integer $k \ge 2$, $\E |Y_i|^k \le k!M^{k-2}v/2$. Then, for every $t > 0$,
\begin{displaymath}
\p\Big(\Big|\sum_{i=1}^p Y_i\Big|\ge t\Big) \le 2\exp\Big(-\frac{t^2}{2(pv + Mt)}\Big).
\end{displaymath}
As a consequence, for every $q \ge 2$,
\begin{align}\label{eq:Bernstein-moments}
\Big\|\sum_{i=1}^p Y_i\Big\|_q \le C(\sqrt{qpv} + qM),
\end{align}
where $C$ is a universal constant.
\end{lemma}

Another (also quite standard) tool we will rely on is the contraction principle for empirical processes due to Talagrand (see Theorem 4.12. in \cite{LedouxTalagrand1991}).
\begin{lemma}[Talagrand's contraction principle]\label{le:contraction} Let $F\colon \R_+ \to \R_+$ be convex and increasing. Let further $\varphi\colon \R \to \R$ be a 1-Lipschitz function such that $\varphi(0)=0$. For every bounded subset $T$ of $\R^n$, if $\varepsilon_1,\ldots,\varepsilon_n$ are i.i.d. Rademacher variables, then
\begin{displaymath}
\E F\Big(\sup_{t\in T} \frac{1}{2} \Big|\sum_{i=1}^n \varphi(t_i)\varepsilon_i\Big|\Big) \le \E F\Big(\sup_{t\in T}\Big|\sum_{i=1}^n t_i\varepsilon_i\Big|\Big)
\end{displaymath}
\end{lemma}

\begin{proof}[Proof of Proposition \ref{prop:empirical-process}]
Let $\varepsilon_1,\ldots, \varepsilon_p$ be i.i.d. Rademacher variables, independent of the sequences $(U_i)$, $(\chi_{i})$.  By the symmetrization inequality (see e.g. Lemma 6.3. in \cite{LedouxTalagrand1991}) we have
\begin{displaymath}
\E W^{q} \le  2^{q} \E \sup_{x \in B_1^n} \Big|\frac{1}{p}\sum_{i=1}^p \varepsilon_i |x^T Z_i|\Big|^{q}.
\end{displaymath}
Now, since the function $t \mapsto |t|$ is a contraction, an application of Lemma \ref{le:contraction}, conditionally on $Z_i$, gives
\begin{align}\label{eq:one_more}
\E W^{q} &\le 2^{2q} \E \sup_{x \in B_1^n} \Big| \frac{1}{p} \sum_{i=1}^p \varepsilon_i x^T Z_i\Big|^{q} = \frac{2^{2q}}{p^{q}} \E \sup_{x \in B_1^n} \Big| x^T \sum_{i=1}^p \varepsilon_i Z_i\Big|^{q}\nonumber \\
&= \frac{2^{2q}}{p^{q}}\E \Big\|\sum_{i=1}^p \varepsilon_i Z_i\Big\|_\infty^{q} = \frac{2^{2q}}{p^{q}}\E \max_{1\le j\le n} \Big|\sum_{i=1}^p \varepsilon_i Z_i(j)\Big|^{q}\nonumber\\
&\le \frac{2^{2q}}{p^{q}}\sum_{j=1}^n \E \Big|\sum_{i=1}^p \varepsilon_i Z_i(j)\Big|^{q}.
\end{align}
Now, for every $i,j$ and every integer $k \ge 2$ we have
\begin{displaymath}
\E |Z_i(j)|^k = \theta \E |U_i(j)|^k \le \theta M^k k! \E e^{|U_i(j)|/M} \le 2 k! \theta M^k = k! v M^{k-2}/2
\end{displaymath}
with $v = 4\theta M^2$. Thus by the moment version \eqref{eq:Bernstein-moments} of Bernstein's inequality for some universal constant $C$ we get
\begin{displaymath}
\E \Big|\sum_{i=1}^p \varepsilon_i X_i(j)\Big|^{q} \le C^{q} \Big( \sqrt{qp\theta}M + qM\Big)^q,
\end{displaymath}
which, when combined with \eqref{eq:one_more}, yields for $q \ge \log n$,
\begin{displaymath}
\|W\|_{q} \le \frac{4Ce}{p}(\sqrt{p\theta q} + q)M.
\end{displaymath}

The first part of the proposition follows by adjusting the constant $C$. The tail bound is a direct consequence of the Chebyshev inequality for the $q$-th moment.
\end{proof}
\bibliographystyle{amsplain}	
\bibliography{ErSpUd}
\end{document}